\documentclass[a4paper,twoside,12pt]{amsart}
\usepackage{amsfonts}
\usepackage{mathrsfs}
\usepackage{amsmath}
\usepackage{amsmath,graphics}
\usepackage{amssymb}
\usepackage{indentfirst,latexsym,bm,amsmath,pstricks,amssymb,amsthm,graphicx,fancyhdr,float,fullpage}
\usepackage[all]{xy}
\usepackage{amsthm}
\usepackage{amscd}
\usepackage[CJKbookmarks=true]{hyperref}
\usepackage{color}
\usepackage{aliascnt}
\usepackage{hyperref}
\usepackage{appendix}
\usepackage{todonotes}

\usepackage{url}

\newtheorem{theorem}{Theorem}[section]

\newtheorem*{acknowledgement*}{\protect\acknowledgementname}

\newaliascnt{setup}{theorem}
\newtheorem{setup}[setup]{Setup}
\aliascntresetthe{setup}

\newaliascnt{question}{theorem}

\aliascntresetthe{question}

\newaliascnt{lemma}{theorem}
\newtheorem{lemma}[lemma]{Lemma}
\aliascntresetthe{lemma}

\newaliascnt{conjecture}{theorem}

\aliascntresetthe{conjecture}

\newaliascnt{proposition}{theorem}
\newtheorem{proposition}[proposition]{Proposition}
\aliascntresetthe{proposition}

\newaliascnt{corollary}{theorem}
\newtheorem{corollary}[corollary]{Corollary}
\aliascntresetthe{corollary}

\newaliascnt{problem}{theorem}

\aliascntresetthe{problem}

\newaliascnt{claim}{theorem}
\newtheorem{claim}[claim]{Claim}
\aliascntresetthe{claim}

\theoremstyle{definition}

\newaliascnt{definition}{theorem}
\newtheorem{definition}[definition]{Definition}
\aliascntresetthe{definition}

\newaliascnt{example}{theorem}

\aliascntresetthe{example}

\theoremstyle{remark}

\newaliascnt{remark}{theorem}
\newtheorem{remark}[remark]{Remark}
\aliascntresetthe{remark}

\newaliascnt{remarks}{theorem}

\aliascntresetthe{remarks}

\def\({$($}
\def\){$)$}

\def\Spec{\textrm{Spec}}

\def\QQ{\mathbb Q}

\newcommand{\Zl}{\mathbb{Z}_{\ell}}

\newcommand{\Qp}{\mathbb{Q}_p}
\newcommand{\Qpbar}{\overline{\mathbb Q}_p}
\newcommand{\Ql}{\mathbb{Q}_{\ell}}
\newcommand{\Qlbar}{\overline{\mathbb{Q}}_{\ell}}

\newcommand{\fisoc}[1]{\textbf{F-Isoc}(#1)}
\newcommand{\fisocd}[1]{\textbf{F-Isoc}^{\dagger}(#1)}

\providecommand{\acknowledgementname}{Acknowledgement}

\def\bQ{{\mathbb Q}}
\def\bZ{{\mathbb Z}}
\def\barQl{{\overline{\bQ}_\ell}}

\def\CC{{\mathbb C}}

\def\OKN{{\mathcal O_K[1/N]}}

\def\X{{\mathfrak X}}
\def\Y{{\mathfrak Y}}
\def\U{{\mathfrak U}}

\def\W{{\mathfrak W}}
\def\D{{\mathfrak D}}

\def\bL{{\mathbb L}}

\def\mL{{\mathcal L}}
\def\mV{{\mathcal V}}
\def\mM{{\mathcal M}}

\def\p{{\mathfrak p}}

\def\GL2{\mathrm{GL}_2}
\def\SL2{\mathrm{SL}_2}
\def \Fq{\mathbb{F}_q}

\def \sHom{\mathcal{H}\mathit{om}}

\numberwithin{equation}{section}

\makeatletter
\@mparswitchfalse
\makeatother
\normalmarginpar

\begin{document}

	\title{Constructing abelian varieties from rank 2 Galois representations}
	\author{Raju Krishnamoorthy}
	\email{krishnamoorthy@alum.mit.edu}
	\address{Arbeitsgruppe Algebra und Zahlentheorie, Bergische Universit\"at Wuppertal, F 13.05 Gaußstraße 20, Wuppertal 42119, Germany}
	\author{Jinbang Yang}
	\email{yjb@mail.ustc.edu.cn}
	\address{School of Mathematical Sciences, University of Science and Technology of China, Hefei, Anhui 230026, PR China 
	\url{https://jbyang1987.github.io/}}
	
	\author{Kang Zuo}
	\email{zuok@uni-mainz.de}
	\address{School of Mathematics and Statistics, Wuhan University, Luojiashan, Wuchang, Wuhan, Hubei, 430072, P.R. China.}
	\address{Institut f\"ur Mathematik, Universit\"at Mainz, Mainz 55099, Germany}
	
	\subjclass[2020]{11G05,  14K15(primary)}
	\keywords{Abelian scheme, Drinfeld modular curve, Langlands program,  Kodaira-Spencer map, Baily-Borel compactification}

\begin{abstract}
Let $U$ be a smooth affine curve over a number field $K$ with a compactification $X$ and let $\bL$ be a rank $2$, geometrically irreducible lisse $\overline{\bQ}_\ell$-sheaf on $U$ with cyclotomic determinant that extends to an integral model, has Frobenius traces all in some fixed number field $E\subset \Qlbar$, and has bad, infinite reduction at some closed point $x$ of $X\setminus U$. We show that $\bL$ occurs as a summand of the cohomology of a family of abelian varieties over $U$. The argument follows the structure of the proof of a recent theorem of Snowden-Tsimerman, who show that when $E=\mathbb Q$, then $\bL$ is isomorphic to the cohomology of an elliptic curve $E_U\rightarrow U$.
\end{abstract}

\maketitle

\tableofcontents

\section{Introduction}
To state our main result, we require the following definition and setup.

\begin{definition} Let $B/k$ be a smooth variety over a finitely generated field and let $\ell\neq \text{char}(k)$ be a prime. An abelian scheme $g\colon A_B\rightarrow B$ is said to be \emph{of $\text{SL}_2$-type} if there is a decomposition of lisse $\Qlbar$-sheaves on $B$:
\begin{equation}\label{eqn:SL2}\displaystyle R^1g_{*}\Qlbar\cong \bigoplus_i \bL_i^{m_i},\end{equation}
where each $\bL_i$ is a geometrically irreducible rank $2$ lisse $\Qlbar$-sheaf on $B$ with cyclotomic determinant: $\bigwedge^2\bL_i\cong \barQl(1)$.
\end{definition}

\begin{setup}\label{setup:main}Let $K$ be a number field and let $X/K$ be a smooth, proper, geometrically irreducible curve. Let $U\subset X$ be a Zariski open and dense subset of $X$ with reduced complementary divisor $D$. Assume that $D$ is non-empty.
\end{setup}

Let $f\colon A_U\rightarrow U$ be a generically simple abelian scheme that is of $\SL2$-type and has bad, infinite reduction along some non-empty subset of $D$. Then the following statements hold for each direct summand $\bL_i$ of $R^1f_{*}\Qlbar$.
\begin{enumerate}
\item $\bL_i$ is a geometrically irreducible, rank 2 lisse $\Qlbar$-sheaf on $U$ with cyclotomic determinant: $\bigwedge^2\bL_i\cong \barQl(1)$.
\item There exists a proper smooth model $\X$ over $\OKN$, an open subset $\U$ of $\X$ extending $U$, an $\ell$-adic local field $M$, and a lisse $\mathcal O_M$-sheaf $\mL_i$ on $\U$ such that $$(\mL_i\otimes_{\mathcal O_M}\Qlbar)|_{U} \cong \bL_i.$$
\item There exists a number field $E$ such that for each closed point $x$ of $\U$, the trace of Frobenius on $(\mL_i)_x$ is in $E\subset \Qlbar$.
\item The local (geometric) monodromy of $\bL_i$ is infinite around some non-empty subset of $D$.
\end{enumerate}
In \cite{snowden2018constructing}, Snowden-Tsimerman prove that when $E=\bQ$, the above four conditions characterize those lisse $\Qlbar$-sheaves coming from families of elliptic curves. More precisely, they prove the following.

\begin{theorem}[Snowden-Tsimerman] Notation as in \autoref{setup:main} and let $\bL$ be a lisse $\Qlbar$-sheaf on $U$ satisfying the above conditions (1)-(4), with $E=\bQ$. Then there exists a family of elliptic curves
$$f: A_U\to U$$
and an isomorphism $\bL \cong {\rm R}^1f_*(\barQl)$.
\end{theorem}

In this article, we consider the situation where Frobenius traces are all contained in a fixed number field $E$.

\begin{theorem}\label{theorem:main} Notation as in \autoref{setup:main} and let $\bL$ be a lisse $\Qlbar$-sheaf on $U$ satisfying conditions (1)-(4). Then there exists an abelian scheme
$$f: A_U\to U$$
such that $\bL$ is a summand $\mathrm{R}^1f_*(\barQl)$.
\end{theorem}

\begin{remark}We view \autoref{theorem:main} as providing a bit of further evidence for the relative Fontaine-Mazur conjecture, as in \cite[Conjecture, p.292]{LZ17} or \cite[Conjecture 1]{petrov2020}.

An observation of Litt implies that for an arithmetic local system, (2) will automatically, hold: see Step 2 of the proof of \cite[ Theorem 1.1.3]{litt2021} or \cite[Theorem 6.1]{petrov2020}. (See also the argument in \cite[Proposition 4.1]{LZ17}.) Therefore, to prove the relative Fontaine-Mazur conjecture for rank $2$ local systems that have infinite monodromy around some point, it suffices to bound the field generated by Frobenius traces. This task seems to be quite difficult in general; for some progress on this question, see \cite{shimizu2020}.
\end{remark}

\begin{remark}
We do not have any idea how to get around point (4). As will be explained in the proof sketch, this is because we crucially use some of Drinfeld's early work on the Langlands correspondence for $\text{GL}_2$ over function fields. More specifically, he is able to show that if $\bL$ is an irreducible rank 2 lisse $\Qlbar$-sheaf over a curve $U/\mathbb{F}_q$ with cyclotomic determinant and infinite monodromy at $\infty$, then $\bL$ comes from a family of abelian varieties over $U$. His proof finds such an abelian scheme as an isogeny factor of a Drinfeld modular curve over $\mathbb{F}_q(U)$.\footnote{One crucial technique he uses is an equal-characteristic $p$ rigid-analytic uniformization of the Drinfeld modular curve at a cusp; indeed, the Drinfeld modular curve will be totally degenerating.} When we do not assume infinite monodromy at $\infty$, then no such result is known; more specifically, the output of his later work on the Langlands correspondence will imply that there exists an open subset $V\subset U\times U$ and a smooth projective morphism $f\colon S\rightarrow V$ of \emph{relative dimension 2} such that $\bL\boxtimes \bL^*|_V$ is a summand of $R^2f_*\Qlbar$. See \cite[Remark 1.4, Question 9.1]{krishnamoorthyrank2} and \cite[Section 1]{krishnamoorthypal2018} for related discussion.
\end{remark}

Our argument largely follows \cite{snowden2018constructing}, but we need several new ingredients. To explain this, we quickly reprise their argument in the following remark.
\begin{remark}[Sketch of \cite{snowden2018constructing}]For notational simplicity, assume that $\mathbb L$ corresponds to a representation:

$$\rho\colon \pi_1(U_K)\rightarrow \mathrm{GL}_2(\mathbb Z_{\ell}),$$
with the property that the mod $\ell^3$ residual representation $\pi_1(U_K)\rightarrow \mathrm{GL}_2(\mathbb Z/\ell^3\mathbb Z)$ is trivial.

\begin{enumerate}
\item Using Drinfeld's first work on the Langlands correspondence over finite fields, for for all $\p\gg 0$ they construct families of elliptic curves over $\U_{\p}$ with trivial $\ell^3$ torsion whose monodromy is isomorphic to $\rho|_{\U_{\p}}$. (This involves an implicit isogeny from what is produced by Drinfeld's theorem.) These families in turn induce maps $$\lambda_{\p}\colon \X_{\p}\rightarrow \bar{\mathcal{M}}_{1,1}(\ell^3)\rightarrow \bar{\mathcal{M}}_{1,1}(\ell^3)\otimes \mathcal O_K/\p,$$ where $\bar{\mathcal{M}}_{1,1}(\ell^3)$ is the compactified moduli space of elliptic curve with full $\ell^3$ level structure, defined over $\text{Spec}(\mathbb Z[1/\ell])$, and the final target is therefore a hyperbolic curve over $\mathcal O_K/\p$.
\item While the map $\lambda_{\p}$ is not a priori generically separable, they factor it through absolute Frobenius to construct a new map, $\mu_{\p}\colon \X_{\p}\rightarrow \bar{\mathcal{M}}_{1,1}(\ell^3)\otimes \mathcal O_K/\p$ which is generically separable. Note that the induced elliptic curve over $\U_{\p}$ also has monodromy isomorphic to $\rho|_{\U_{\p}}$. Then Riemann-Hurwitz applies, bounding the degree of the map $\mu_{\p}$ by some number $d$, which is crucially independent of $\p$. We may replace $\lambda_{\p}$ with $\mu_{\p}$.
\item At this point, consider the moduli space of maps:

$$\mathcal H:= \text{Hom}^{\leq d}_{\OKN}(\X,\bar{\mathcal M}_{1,1}(\ell^3)),$$
of morphisms of curves $\lambda$ over $\OKN$, with degree bounded by $d$. This moduli space is a scheme of finite type over $\OKN$ because we have put a bound on the degree.\footnote{In fact, potentially at the cost of increasing $N$, this moduli space is \emph{finite flat} over $\OKN$. However, this fact is not used in their proof.} For each $k$, let $\mathcal H_k$ denote the subset of $\mathcal H$ consisting of those maps $\lambda$ such that:

\begin{itemize}
\item $\lambda(\U)\subset \mathcal M_{1,1}(\ell^3)$;
\item and the induced elliptic curve $E_U\rightarrow U$ has mod $\ell^k$ monodromy isomorphic to $\rho\text{ mod }\ell^k$.
\end{itemize}
Then Snowden-Tsimerman show that each the $\mathcal H_k$ is a closed subset, and hence so is $\mathcal H_{\infty}:=\cap \mathcal H_k$. The subset $\mathcal H_\infty\subset \mathcal H$ will parametrize those maps $\lambda$ such that the monodromy representation is \emph{integrally} isomorphic to $\rho$. Equipping $\mathcal H_{\infty}\subset \mathcal H$ with the reduced induced subscheme structure, they deduce that $\mathcal H_{\infty}$ is therefore a scheme of finite type over $\OKN$. As it has points modulo $\p$ for infinitely many primes $\p$ of $\mathcal O_K$, it follows that it has a point over a finite extension field $K'/K$. Then a Weil restriction argument, together with Faltings' isogeny theorem, allows one to conclude.
\end{enumerate}

\end{remark}

We now explain the new ingredients in turn, highlighting the additional difficulties.
\begin{remark}[Sketch of the proof of Theorem \ref{theorem:main}] Again for notational simplicity, assume that $\mathcal L$ corresponds to a representation:

$$\rho\colon \pi_1(U_K)\rightarrow \mathrm{GL}_2(\mathbb Z_{\ell}),.$$
(Note that $\mathbb Q_{\ell}$ contains number fields of infinitely large degree!) We further assume that $\rho$ has the property that the mod $\ell^3$ residual representation $\pi_1(U_K)\rightarrow \mathrm{GL}_2(\mathbb Z/\ell^3\mathbb Z)$ is trivial. (This last assumption will play no role, but we include it to see which additional technicalities emerge.)

\begin{enumerate}
\item Again using Drinfeld's early work on the Langlands correspondence over finite fields, for each $\p\gg 0$, we may construct an abelian scheme over $f_{\p}\colon A_{\p}\rightarrow \U_{\p}$ of dimension $h:=[E:\bQ]$, such that $\mL|_{\U_{\p}}$ injects as a summand of $\mathrm{R}^1f_{\p,*}\Qlbar$.

\begin{enumerate}
\item Here we encounter our first complication: it is \emph{not necessarily true} that we can choose $A_{\p}[\ell^3]$ to be the split \'etale cover of $\U_{\p}$: unlike in the case \cite{snowden2018constructing}, $\mL|_{\U_{\p}}$ is not all of $\mathrm{R}^1f_{\p,*}\Qlbar$. However, there exists a \emph{finite, connected} cover $\varphi_{\p}\colon (\X_{\p})'\rightarrow \X_{\p}$ (purely in characteristic $p$!) of degree $\leq |\mathrm{GL}_{2h}(\bZ/\ell^3\bZ)|$ such that:
\begin{itemize}
\item the map $\varphi_{\p}$ is finite \'etale over $\U_{\p}$;
\item if we set $(\U_{\p})':=\varphi_{\p}^{-1}(\U_{\p})$, the pullback $A'\rightarrow (\U_{\p})'$ has trivial $\ell^3$-torsion;
\item the abelian scheme $f'_{\p}\colon A'\rightarrow (\U_{\p})'$ has semi-stable reduction at $(\X_{\p})'$.
\end{itemize}
The key property of the cover $\varphi\colon (\X_{\p})'\rightarrow \X_{\p}$ is that the degree is bounded independent of $\p$. However, we emphasize that, as of yet, there is no preferred $X'\rightarrow X$ over $K$ that patches all of these modulo $\p$ covers together. At this point, we demand that $N>|\mathrm{GL}_{2h}(\bZ/\ell^3\bZ)|$, to ensure that any such $\varphi_\p$ is tamely ramified.

\item We now encounter our next (minor) trouble. A priori, there is no bound on the degree of the polarization of $f'_{\p}\colon A'_{\p}\rightarrow (\U_{\p})'$. This has a simple solution: Zarhin's trick, which says that $B'_{\p}:=(A'_{\p}\times (A'_{\p})^t)^4$ has a principal polarization.

\item There is a third trouble; unlike in the approach of Snowden-Tsimerman, we have not yet nailed down the integral monodromy, and this is more subtle. There are several ways one could address this. Our solution to this problem will be found in the construction of a simple moduli space, $\mathcal H_k$: see Step (3).

\end{enumerate}
We therefore get a map:

$$\lambda'_{\p}\colon (\X_{\p})'\rightarrow \mathcal A^*_{8h,1,\ell^3}\rightarrow \mathcal A^*_{8h,1,\ell^3}\otimes \mathcal O_K/\p,$$
where $\mathcal A^*_{8h,1,\ell^3}$ is the Baily-Borel compactification of the fine moduli scheme $\mathcal A_{8h,1,\ell^3}$ parametrizing principally polarized abelian schemes of dimension $8h$ and trivial level $\ell^3$ structure. This $\lambda'_{\p}$ has the following property: the pullback of the universal rank $16h$ lisse $\ell$-adic sheaf on $\mathcal A_{8h,1,\ell^3}$ to $\U_{\p}$ has $\rho$ as a \emph{rational} summand.

\item Our next goal is to somehow numerically bound $\lambda'_{\p}$. Recall that \cite{snowden2018constructing} do this by a combination of Riemann-Hurwitz and factoring through some power of absolute Frobenius. In our setting, this step is more tricky, and we chose to use an Arakelov-style inequality. More precisely, if $f_\p: \bar{B'}_{\p}\rightarrow (\X_{\p})'$ is the N\'eron model of $B'_{\p}\rightarrow \U_{\p}',$ then we will bound the degree of
the Hodge vector bundle $E^{1,0}_{(\X_{\p})'}:=R^0{f_\p}_*\Omega^1_{\bar {B'}_{\p}/ (\X_\p)' }(\log \Delta)$, at least for many infinitely many $\p$. Set
$E^{0,1}_{(\X_{\p})'}:=R^1{f_p}_*\mathcal O_{ \bar{B'}_{\p} }$. Then to bound the degree of $E^{1,0}_{(\X_{\p})'}$, we will need to know that \emph{the logarithmic Kodaira-Spencer map} constructed by Faltings-Chai,
$$\theta'_{\p}: E^{1,0}_{(\X_{\p})'}\to E^{0,1}_{(\X_{\p})'}\otimes \Omega^1_{(\X_{\p})'}(\log D)$$
is not only nonzero but is moreover an isomorphism at the generic point.\footnote{The logarithmic Kodaira-Spencer map $\theta'_p$ is the derivative of the period map $\lambda'_{\p}$.
	We emphasize: knowing that map $\lambda'_{\p}$ is generically separable (equivalently, $\theta'_{\p}$ being nonzero) is not sufficient for our application. Morally speaking, we need to prove that the image of $\lambda'_{\p}$ is not contained in one of the natural foliations of the relevant Hilbert modular variety. However, we chose to not work with good reductions of Hilbert modular varieties, as that theory is more complicated than the bare theory of moduli of abelian varieties.} In more detail:
For any $\p\gg 0$ such that the underlying prime number $p$ splits completely in $E$, the field generated by Frobenius traces, the induced $p$-divisible group on $(\U_{\p})'$ splits as the direct sum of several copies of $h$ (mutually non-isogenous) height $2$, dimension $1$ $p$-divisible groups $G'_i$ and their duals $(G'_i)^t$. We prove, using monodromy considerations, that they are generically ordinary and have supersingular points. Applying a Frobenius untwisting lemma from the PhD thesis of Jie Xia \cite{xia2013deformation}, we conclude that we may ``Frobenius untwist'' each of them until they are all generically versally deformed.\footnote{This mimics the elliptic modular setting for the following reason: the equal-characteristic universal deformation space of a height $2$, dimension $1$ $p$-divisible group is one dimensional.} (In the appendix, we provide a proof of Xia's Frobenius untwisting lemma in our context, and also give a second argument and perspective of the termination of Frobenius untwisting stability techniques.) Once again using Zarhin's trick, we will obtain an isogenous, principally polarized abelian scheme over $\U_{\p}'$, which we relabel $B'_{\p}$, with the N\'eron model
$$ f_\p: \bar {B'}_{\p}\to (\X_\p)'$$ and such that the logarithmic Kodaira-Spencer map is a generically injective map of coherent sheaves on $(\X_{\p})'$. By taking determinants, we deduce an Arakelov-style inequality, thereby bounding the degree of the induced Hodge line bundle on $(\X_{\p})'$ by some integer $d$, which is crucially \emph{independent of $\p$}. The output of this is Lemma \ref{lemma:main_char_p}.

\item To mimic the third step, we first construct some finite type moduli spaces of $\OKN$, and then we use our argument as above to show it has points modulo $\p$ for infinitely many $\p$. In greater detail:

\begin{enumerate}
\item Fix $d>1$ and set $\mathcal H$ to be the moduli of triples $(\X',\varphi,\lambda)$:
\entrymodifiers={+!!<0pt,\fontdimen22\textfont2>}
$$\xymatrix@1{
\X'\ar[rr]^-\lambda\ar[d]_{\varphi}	&& \mathcal A^*_{8h,1,\ell^3}\otimes \OKN\\
	\X,	&&
}$$
where
\begin{itemize}
\item $\X'/\OKN$ is a smooth, proper, geometrically connected curve;
\item $\varphi$ is finite, of degree at most $\leq |\mathrm{GL}_{16h}(\bZ/\ell^3\bZ)|$, and \'etale over $\U$;
\item there exists some point $\infty'\in \X'$ that is sent to a 0-dimensional cusp in $\mathscr A^*_{8h,1,\ell^3}$;
\item and the degree of the pulled-back Hodge line bundle on $\X'_K$ is $\leq d$.
\end{itemize}
Then $\mathcal H$ will be a Deligne-Mumford stack, of finite type over $\OKN$. (The stackiness comes from the intervening Hurwitz space.) We further show that the generic fiber of $\mathcal H$ has dimension 0 (and is in fact reduced); while this is an immediate corollary of the PhD thesis of Ben Moonen, we argue following the work of Saito. This implies that, after potentially further increasing $N$, the stack $\mathcal H/\OKN$ has relative dimension 0.\footnote{By specialization of the prime-to-$p$ fundamental group, the fact that $\mathcal H$ has relative dimension 0 over $\OKN$ will imply that the geometric local system $\bL|_{U_{\bar K}}$ comes from a family of abelian varieties. This does not use any of the more delicate moduli spaces $\mathcal H_k$ to come.}
\item Recall that $\mathcal L$ is a lisse $\Zl$-sheaf on $\U$, whose generic fiber is an $\Zl$-lattice inside of $\bL$. For $k\geq 1$, set $\mathcal H_k$ to be the subspace of $\mathcal H$ given by those $(\X',\varphi,\lambda)$ (with induced abelian scheme $f\colon B'\rightarrow \X'$) such that there exists a map
$$\psi\colon \varphi^*(\mathcal L)/\ell^k\rightarrow R^1f_*\mathcal \Zl/\ell^k$$
of torsion locally constant \'etale sheaves with the following condition: the reduction modulo $\ell$ of $\psi$ is nonzero. (This condition is crucial in our approach.)\footnote{When the lisse $\Qlbar$-sheaf has coefficients in an $\ell$-adic local field $M$, not necessarily assumed to be $\Ql$, then we instead demand that there exists such a $\psi$ whose reduction modulo a fixed uniformizer $\pi_M$ is non-zero.} Then $\mathcal H_k\subset \mathcal H$ will be a closed substack, which we may equip with the reduced induced structure. Similarly, set $\mathcal H_{\infty}$ to be $\cap \mathcal H_k$, again with the reduced induced stack structure. Note that $\mathcal H_{\infty}$ is then a finite type Deligne-Mumford stack over $\text{Spec}(\OKN)$ for some $N$.

Unlike in the Snowden-Tsimerman approach, the relationship of the moduli space $\mathcal H_{\infty}$ to Drinfeld's theorem is not immediately apparent. However, in both approaches, the moduli spaces involve extra maps of \emph{$\ell^k$-torsion sheaves} rather than lisse $\Zl$-sheaves.
\item By the careful choice of $\mathcal H_k$ and a crucial diagonalization argument on $\mathcal H_{\infty}$ (contained in Lemma \ref{remark:Hk}), it will follow from the earlier steps that there exists an infinite set of primes $\p$ such that $\mathcal H_{\infty}$ has points modulo $\p$. (Unlike the approach of Snowden-Tsimerman, this does not require one take an $\ell$-primary isogeny.) By the Nullstellensatz, one deduces that $\mathcal H_{\infty}$ has characteristic 0 points. A Weil restriction argument then yields the result.
\end{enumerate}
\end{enumerate}
\end{remark}

\begin{remark} Katz has shown that rigid local systems on the punctured projective line are motivic, and Corlette-Simpson have shown that all rigid rank 2 local systems are motivic. Our main theorem provides a new arithmetic approach to both Katz's theorem in rank 2 and also the Corlette-Simpson theorem, subject to an additional assumption analogous to $(4)$. Here is an outline of the proof. We emphasize that these approaches will critically rely on a quasi-projective version of a deep theorem on projective varieties of Esnault-Groechenig \cite{egrigid}, which is not yet available.

First we assume that $U$ is a curve. Let $\bL$ be a cohomologically rigid local system of rank $2$ on $U^{an}_{\mathbb C} $ with coefficients in $\Qlbar$, trivial determinant and infinite monodromy around $\infty$. Suppose that the local system $\bL$ spreads out to an \'etale local system $\mathcal{L}$ with cyclotomic determinant over a finitely generated spreading out $\mathfrak{U}/S$ such that the \emph{stable Frobenius trace fields} are bounded, i.e., there exists a number field $E$ such that for all closed points $s$, there exists a finite extension $s'/s$ such that the Frobenius trace field of $\mathcal{L}|_{\mathfrak U_{s'}}$ is contained in $E$. Then, our argument applies verbatim to prove that $\bL$ over $U_{\mathbb C}$ comes from a family of abelian varieties; we get mod $p$ points for infinitely many $p\gg 0$, and the relevant moduli space is of finite type and in fact generically 0 dimensional, so by specialization of the prime-to-$p$ fundamental group we may conclude.

In fact, recent work of the first-named author and J.~Lam shows the following. If $X/\mathbb C$ is a projective variety, and if $\bL$ is a cohomologically rigid $\Qlbar$-local system with trivial determinant on $X^{an}$, then there exists a spreading out $\mathfrak X/S$ and a number field $E\subset \Qlbar$ such that $\bL$ spreads out to an \'etale local system $\mathcal{L}$ on $\mathfrak X$ with cyclotomic determinant such that the stable Frobenius trace field is contained inside of $E$. In the quasi-projective case, which is what is needed here, the arguments of \emph{loc. cit.} should go through once a quasi-projective analog of the theorem of \cite{egrigid} is available: we need that cohomologically rigid stable flat connections give rise to $F^f$-isocrystals on the relevant $p$-adic completions.

In general, when $U$ is higher dimensional (i.e., $U=X\setminus D$, where $X$ is a smooth projective variety and $D$ is a simple normal crossings divisor), it is very plausible that once the quasi-projective analog of \cite{egrigid} is available, one may similarly deduce the analog of the Corlette-Simpson theorem here (again, subject to the restriction that the local monodromy around one of the boundary divisors is infinite). Here is a sketch of the argument. The putative quasi-projective version of \cite{egrigid}  will in fact output rank 2 \emph{filtered logarithmic $F$-crystals}; as above, porting these objects into the techniques of Krishnamoorthy-Lam, one can deduce that cohomologically rigid rank 2 local systems will have spreading-outs whose stable Frobenius trace field is bounded. A complete set of companions to the logarithmic $F$-isocrystals so constructed will likely exist, as in the projective case this is shown in \cite{egrigid}. From these $F$-isocrystals, \cite{krishnamoorthypal2022} will provide abelian schemes on open subsets of the mod $p$ fibers of bounded dimension\footnote{the boundedness of the dimension comes from the boundedness of the stable Frobenius trace field} and \cite[Corollary 6.12]{krishnamoorthypal2018} shows that, after possibly replacing with an isogenous abelian scheme, the abelian schemes extend to the  whole mod $p$ fiber of $\mathfrak U$. We can bound the degree of the Hodge line bundle for infinitely many $p$ by Frobenius untwisting, exactly as in done here. Finally, the appropriate Hom scheme will again be 0-dimensional, so by using specialization of the prime-to-$p$ fundamental group one may again conclude.
\end{remark}

\section{Drinfeld's work on the Langlands correspondence for $\mathrm{GL}_2$ and some corollaries}
A key ingredient in the proof of \autoref{theorem:main} is the following \autoref{Theorem:GL2}, which is a byproduct of Drinfeld's first work on the Langlands correspondence for $\mathrm{GL}_2$. We first record a setup.
\begin{setup}\label{setup:curve_finite_field}
Let $p$ be a prime number and let $q=p^a$. Let $C/\Fq$ be a smooth, affine, geometrically irreducible curve with smooth compactification $\bar{C}$. Let $Z:=\bar{C}\setminus C$ be the reduced complementary divisor.
\end{setup}

\begin{theorem}\label{Theorem:GL2}(Drinfeld) Notation as in \autoref{setup:curve_finite_field} and let $\bL$ be a rank 2 irreducible $\Qlbar$ sheaf on $C$ with determinant $\Qlbar(1)$. Suppose $\bL$ has infinite local monodromy around some point at $\infty\in Z$. Then $\bL$ comes from a family of abelian varieties in the following sense: let $E$ be the field generated by the Frobenius traces of $\bL$ and suppose $[E:\QQ]=h$. Then there exists an abelian scheme
\[
\pi_C\colon A_{C}\rightarrow C
\]
of dimension $h$ and an isomorphism $E\cong \textrm{End}_{C}(A)\otimes\QQ$, realizing $A_C$ as a $\SL2$-type abelian scheme, such that $\bL$ occurs as a summand of $R^1(\pi_C)_*\Qlbar$. Moreover, $A_{C}\rightarrow C$ is totally degenerate around $\infty$.
\end{theorem}

See \cite[Proof of Proposition 19, Remark 20]{snowden2018constructing} for how to recover this result from Drinfeld's work. This amounts to combining \cite[Main Theorem, Remark 5]{drinfeld1983} with \cite[Theorem 1]{drinfeld1977}.

We make some observations about the $p$-adic properties of the resulting abelian schemes. In particular, our goal is to show that, in the context of \autoref{Theorem:GL2}, we can modify $A_C\rightarrow C$ with products, duals, and isogenies such that the resulting abelian scheme $B_C\rightarrow C$ that has especially nice ($p$-adic) properties; these will in turn allow us to prove an Arakelov-style inequality. First, we will give the following non-standard definition, which is adapted for our purpose.
\begin{definition}\label{definition:strong_strict_ss}Maintain notation as in \autoref{setup:curve_finite_field}. Let $G_C$ be a $p$-divisible group on $C$. We say $G_C$ has \emph{strong strict semistable reduction along $Z$} if
\begin{itemize}
\item $G_C$ has semistable reduction along $Z$ \cite[Definition 4.2]{trihanBT} (which is based on semistable reduction in the sense of de Jong \cite[Definition 2.2]{de1998homomorphisms}); and
\item if for every point $\infty \in \bar{C}\backslash C$ with local parameter $z_{\infty}$ the restricted $p$-divisible group $$G_C|_{\text{Spec}(\Fq((z_{\infty})))}$$ over $\text{Spec}(\Fq((z_{\infty})))$ does not extend to a $p$-divisible group over $\text{Spec}(\Fq[[z_{\infty}]]).$
\end{itemize}
\end{definition}

\autoref{definition:strong_strict_ss} is useful as it concisely expresses the condition that $G_C$ have semistable reduction and moreover that it does not extend as a $p$-divisible group across \emph{any} of the cusps.

The next proposition will be critical for bounding degrees of maps to moduli spaces. In the appendix, we explain a second proof/perspective of the second part, which is based on a destabilizing iteration argument due to Langer.

\begin{proposition}\label{prop:vers_def}
Maintain notation as in \autoref{setup:curve_finite_field}. Let $G_C$ be a height $2$, dimension $1$ $p$-divisible group on $C$ with strong strict semistable reduction along $Z$. Suppose further that $\mathbb D(G_C)\otimes \Qpbar$ is an irreducible object of $\fisocd{C}_{\Qpbar}$. Then
\begin{enumerate}
\item $G_C$ is generically ordinary with a non-empty supersingular locus; and
\item there exists an isogenous $p$-divisible group $H_C\rightarrow G_C$ that is \emph{generically versally deformed} (in the sense of \cite[Defintions 8.1, 8.2]{krishnamoorthyrank2}).\footnote{We briefly recall the notion here. Let $G_C\rightarrow C$ be a height $2$, dimension $1$ $p$-divisible group. There is a Kodaira-Spencer map: $KS\colon T_C \rightarrow \Omega^*\otimes \Psi$, where $\Omega$ is the Hodge line bundle of $G_C$ and $\Psi$ is the dual of the Hodge line bundle of the Serre dual $G^t_C$. We say that $G_C\rightarrow C$ is \emph{generically versally deformed} if the above $KS$ is nonzero. After the work of Illusie \cite{illusie1985deformations}, this is equivalent to the following condition: there exists a closed point $c$ such that the map $u_c\colon C^{\wedge}_c\rightarrow \text{Def}(G_c)$ from the formal completion of $C$ at $c$ to the equal-characteristic universal deformation space of $G_c$ is a formally smooth map of formally smooth, $1$-dimensional $\kappa(c)$ schemes, i.e., $u_c$ is an isomorphism.}
\end{enumerate}
\end{proposition}
Before we begin the proof, we comment on the overconvergence assumption. If $H_C\rightarrow C$ is a $p$-divisible group, then $F$-isocrystal $\mathbb{D}(H_C)$ is automatically a convergent $F$-isocrystal. In our setting, the fact that we demand $G_C\rightarrow C$ to be semistable around $Z$ implies that $\mathbb{D}(G_C)$ is in fact overconvergent. Part of the hypothesis of \autoref{prop:vers_def} is then that $\mathbb{D}(G_C)\otimes \Qp$ is absolutely irreducible in $\fisocd{C}$.
\begin{proof}[First proof of \autoref{prop:vers_def}]
As $G_C$ has height $2$ and dimension $1$, there are only two possible Newton polygons, which correspond to the $p$-divisible group being ordinary or supersingular respectively. If $G_C$ were not generically ordinary, it would be everywhere supersingular. However, supersingular $p$-divisible groups cannot be strictly semistable: as there is no multiplicative part, the filtration in \cite[Definition 2.2]{de1998homomorphisms}, would have to be trivial, which would imply that $G_C$ extends to a $p$-divisible group over $\bar{C}$. This shows $G_C$ is generically ordinary.

Suppose that $G_C$ had no supersingular points: then $G_C$ is everywhere ordinary. Let $H_C$ be the multiplicative sub $p$-divisible group of $G_C$, i.e., the height $1$, dimension $1$ $p$-divisible group with Newton slope $1$ everywhere. Let $\infty \in Z$, with formal parameter $z_{\infty}$. Then the $p$-divisible group $G_C|_{\text{Spec}(\Fq((z_{\infty})))}$ has semistable reduction in the sense of \cite[Definition 2.2]{de1998homomorphisms} and does not extend to a $p$-divisible group over $\text{Spec}(\Fq[[z_{\infty}]])$. Then, for the definition of semi-stability to be satisfied, the only possible filtration is:

$$G^{\mu}_{\text{Spec}(\Fq((z_{\infty})))} = G^f_{\text{Spec}(\Fq((z_{\infty})))} = H_C|_{\text{Spec}(\Fq((z_{\infty})))}.$$
(Here, the meaning of $G^{\mu}_{\text{Spec}(\Fq((z_{\infty})))}$ and $G^f_{\text{Spec}(\Fq((z_{\infty})))}$ is given as in \cite[Definition 2.2]{de1998homomorphisms}.)

However, by definition of semi-stability, $H_C|_{\text{Spec}(\Fq((z_{\infty})))}$ therefore extends to a $p$-divisible group over $\text{Spec}(\Fq[[z_{\infty}]])$. Ranging over all points $\infty \in Z$, we see that $\mathbb D(H_C)\otimes \Qp \in \fisoc{C}$ in fact extends to an $F$-isocrystal on $\fisoc{\bar{C}}$: therefore $\mathbb D(H_C)\otimes \Qp\in \fisocd{C}$. However, this yields a sub-object (in $\fisocd{C}$) of $\mathbb D(G_C)\otimes \Qp$, contradicting the absolute irreducibility of the hypothesis. Therefore $G_C$ has a non-empty supersingular locus.

Now, suppose that $G_C\rightarrow C$ is not generically versally deformed, i.e., that $KS=0$ identically on $C$. Then by \cite[Theorem 6.1]{xia2013deformation}, there is a $p$-divisible group $(G_1)_C\rightarrow C$ such that $(G_1)_C^{(p)}\cong G_C$, i.e., the Frobenius twist of $(G_1)_C$ is isomorphic to $G_C$. The $p$-divisible groups $G_C$ and $(G_1)_C$ are isogenous. If the Kodaira-Spencer map for $(G_1)_C$ is nonzero, we may stop. Otherwise, we may apply \cite[Theorem 6.1]{xia2013deformation} again to find a $p$-divisible group $(G_2)_C\rightarrow C$ such that $(G_2)_C^{(p)}\cong (G_1)_C$. We claim this process must terminate at some point. Here is a simple proof (also indicated in \cite[p. 253]{krishnamoorthyrank2}). Let $c$ be a closed point of $C$ such that $G_c$ is supersingular. Then the (equal characteristic) deformation map:

$$u_c\colon C^{\wedge}_{/c}\rightarrow \text{Def}(G_c)\cong \kappa(c)[[t]]$$
is nonzero, because $G_C$ is generically ordinary. (In other words, if $u_c$ were 0, then the $p$-divisible group over $\text{Spec}(\kappa(c)[[z_c]])$ would be supersingular at both the closed and the generic point, which is a contradiction: over the generic point, the $p$-divisible group is base-changed from $\Fq(C)$ along the map $\Fq(C)\hookrightarrow \kappa(c)[[z_c]]$.)

The map $KS_c$ is simply the derivative of $u_c$. In particular, $KS_c=0$ implies that $u_c^*(t)\in \kappa(c)[[z_c]]$ is a power series in $z_c^p$; on the level of the universal deformation map $u_c$, Frobenius untwisting amounts to extracting a $p^{\text{th}}$ root of $u_c^*(t)$. As $u_c$ is not constant, this process must terminate.
\end{proof}

\begin{corollary}\label{corollary:isogenous_av}
Notation as in \autoref{Theorem:GL2}. Suppose further the following.

\begin{itemize}
\item The lisse $\barQl$-sheaf $\bL$ has infinite, unipotent local monodromy around \emph{each} point $\infty\in Z$.
\item Let $E$ be the field generated by Frobenius traces of $\bL$. Suppose that $p$ splits completely in $E$.
\end{itemize}

Then there exists an abelian scheme $f_C\colon A_C\rightarrow C$ satisfying all of the conclusions of Drinfeld's \autoref{Theorem:GL2}, together with the following further properties.

$$\displaystyle A_C[p^{\infty}]\cong \bigoplus_i G_{C,i},$$
where
\begin{enumerate}
\item the $G_{C,i}$ are all mutually non-isogenous;
\item each $G_{C,i}$ is a height $2$, dimension $1$ $p$-divisible group on $C$; and
\item each $G_{C,i}$ is generically versally deformed, generically ordinary, and has non-empty supersingular locus.
\end{enumerate}
\end{corollary}
\begin{proof}
We will first construct $\bigoplus_i G_{C,i}$ with the desired properties.

Let $f\colon A_C\rightarrow C$ be an abelian scheme produced by Drinfeld's \autoref{Theorem:GL2}. By Grothendieck's monodromy criterion for semistable reduction, $A_C\rightarrow C$ is totally degenerate around every point of $Z$. The $F$-isocrystal $\mathcal{E}:=\mathbb D(A_C[p^{\infty}])\otimes \Qp$ is a semi-simple object of $\fisocd{C}$ by \cite{pal2015monodromy}. We claim that $\mathcal E$ is the companion to $R^1(\pi_C)_*\Qlbar$. Indeed, a theorem of Zarhin \cite[Chapitre XII, Theorem 2.5,
p. 244-245]{moret1985pinceaux} implies that $R^1(\pi_C)_*\Qlbar$ is semi-simple and the characteristic polynomials of Frobenius agree at closed points by \cite{katzmessing}. On the other hand, there is a decomposition:
\begin{equation}\label{eqn:decom_Fisoc}\displaystyle \mathcal{E}_{\Qpbar}\cong \bigoplus_i \mathcal E_i,\end{equation}
where $\mathcal E_i$ are \emph{irreducible} objects of $\fisocd{C}_{\Qpbar}$

It follows from \cite[Remark 2.8]{krishnamoorthypal2022} that every summand $\mathcal E_i$ is a companion of $\bL$.\footnote{Strictly speaking, the remark only states this for the $\ell$-adic companion, but the preceding sentence then follows immediately from the theory of companions.} As the relation of companions preserves ``infinite monodromy at $\infty\in Z$'', each $\mathcal E_i$ has infinite monodromy around every $\infty\in Z$.

 Also, $\det(\mathcal E_i)=\Qpbar(1)$, again because the property ``cyclotomic determinant'' is preserved under the companions relation.

As $p$ splits completely in $E$, it follows that $E\otimes \Qp\cong \Pi \Qp$ acts on $\mathcal E$, and the images of the idempotents are the $\mathcal E_i$, i.e., the (absolutely) irreducible summands $\mathcal E_i$ are objects of $\fisocd{C}$.

The slopes of each $\mathcal E_i$ are in between $0$ and $1$. Therefore we may apply \cite[Lemma 5.8]{krishnamoorthypal2018}\footnote{this is really due to Katz, see \cite[Theorem 2.6.1]{katz1979slope}.} and \cite{de1995crystalline} to see that for each $\mathcal E_i$, there exists a (non-canonical) $p$-divisible group $G_{C,i}$ with $\mathbb D(G_{C,i})\otimes \Qpbar \cong \mathcal E_i$. (Equivalently, note that $\mathcal E_i\in \fisocd{C}$, i.e., each $\mathcal E_i$ has coefficients in $\Qp$, by the hypothesis that $p$ splits completely in $E$.)

 The $p$-divisible groups $A_C[p^{\infty}]$ and $\bigoplus_i G_{C,i}$ are isogenous. At this point, we wish to claim that each $G_{C,i}$ has strong strict semistable reduction along $Z$. First of all, note that $A_C[p^{\infty}]$ has has strong strict semistable reduction by \cite[2.5]{de1998homomorphisms}.

 As $\mathbb D(G_{C,i})\otimes \Qp$ is overconvergent, it follows from \cite[Theorem 2.22]{pal2015monodromy} that every $G_{C,i}$ has semistable reduction along $Z$. Suppose for contradiction that $G_{C,1}$ extended through some cusp $\infty \in Z$. Then $\mathcal E_1 \cong \mathbb{D}(G_{C,1})\otimes \Qp$ also extends to an (overconvergent) $F$-isocrystal on the curve $C\cup\{\infty\}=\bar{C}\setminus \{Z\setminus \infty\}$. As each of the $\mathcal E_i$ are companions, this implies that they all also extend to $C\cup\{\infty\}$. Therefore the $\ell$-adic companion also extends to $C\cup\{\infty\}$. This implies that $\bL$ also extends to a lisse $\Qlbar$-sheaf on $C\cup\{\infty\}$, contradicting our assumption that $\bL$ had infinite, unipotent monodromy around $\infty$.

 We may now apply \autoref{prop:vers_def} to replace each $G_{C,i}$ with an isogenous $p$-divisible group that satisfies the two conclusions of the proposition. Note that we still have the relation: $A_C[p^{\infty}]$ is isogenous to $\bigoplus G_{C,i}$.

By \cite[Lemma 2.13]{krishnamoorthypal2022}, it follows that we can replace $A_C$ by an isogenous abelian scheme such that:

$$\displaystyle A_C[p^{\infty}]\cong \bigoplus_i G_{C,i},$$
where every $G_{C,i}$ is generically versally deformed, is generically ordinary, and has supersingular points. Finally, each $G_i$ will be mutually non-isogenous because the $F$-isocrystals $\mathbb D(G_i)\otimes \Qp$ are a complete collection of $p$-adic companions of $\mathbb L$ \cite[Remark 2.8]{krishnamoorthypal2022}.

\end{proof}

Using the above, we will be able to extract all of the $p$-adic information we need from \autoref{Theorem:GL2} to prove \autoref{theorem:main}. We need one final piece of notation.
\begin{definition}
Let $N\geq 1$ be an integer prime to $p$ and let $g\geq 1$ be a positive integer. Then $\mathscr{A}_{g,1,N}/\Spec(\bZ[1/N])$ denotes the (fine) moduli space of principally polarized abelian varieties with trivial full level $N$ structure. This is a smooth, quasi-projective scheme over $\Spec(\bZ[1/N])$. It has a compactification, $\mathscr{A}^*_{g,1,N}/\Spec(\bZ[1/N])$.\footnote{The moduli space $\mathscr{A}_{g,1,N}$ is not geometrically connected over $\Spec(\mathbb Z[1/N])$. This is why some authors prefer to work with the geometrically connected components, which are defined over $\Spec(\mathbb Z[\zeta_{N},1/N])$.} This latter scheme has a natural ample line bundle, the \emph{Hodge line bundle}, which we denote by $\alpha$.
\end{definition}
Then the precise output we need from Drinfeld's \autoref{Theorem:GL2} is given in the following lemma.

\begin{lemma}\label{lemma:main_char_p}
Notation as in \autoref{Theorem:GL2}. Suppose that $p$ splits completely in $E$ and $\bL$ has infinite, unipotent monodromy around every point of $Z$.

Then there exists a principally polarized abelian scheme abelian scheme $f\colon B_C\rightarrow C$, of $\mathrm{SL}_2$-type and dimension $8h$, such that $\bL$ occurs as a direct summand of $R^1f_*\overline{\bQ}_\ell$, and the following holds.
\begin{enumerate}

\item $B_C\rightarrow C$ has semistable, infinite reduction along $\bar{C}\backslash C$. Call the N\'eron model $\bar{B}_{\bar C}\rightarrow \bar{C}$.
\item There exists $h$ mutually non-isogenous $p$-divisible groups $G_{C,i}$, each of height $2$ and dimension $1$ and generically versally deformed, such that there is a decomposition of $p$-divisible groups $$\displaystyle B_C[p^{\infty}]\cong \bigoplus_i (G_{C,i}\times G_{C,i}^t)^4.$$
\item After Kato-Trihan, to $\bar{f}\colon \bar{B}_{\bar C}\rightarrow \bar{C}$ there is an associated logarithmic $F$-crystal with nilpotent residues $(M,F)$ in finite, locally free modules on the log pair $(\bar C,Z)$. Similarly, there is a \emph{logarithmic Hodge vector bundle}, which we write $\Omega_{\bar{B}/\bar{C}}$, a rank $8h$ vector bundle on $\bar{C}$.Then the following hold.

\begin{itemize}

\item[(i)] $\Omega_{\bar{B}/\bar C}$ splits as the direct sum of $8h$ positive line bundles on $\bar{C}$;
\item[(ii)] the log Kodaira-Spencer map (constructed in \cite[Ch. III, Cor 9.8]{faltings2013degeneration}:
$$\theta\colon \Omega_{\bar B/\bar C}\rightarrow R^1\bar{f}_* (\mathcal O_{\bar B})\otimes \omega_{\bar C}(Z),$$
where $\omega_{\bar C}$ denotes the sheaf of differential one-forms on $\bar{C}$, is an injective map of coherent sheaves on $\bar{C}$;

\item[(iii)] $\deg(\Omega_{\bar{B}/\bar C}) \leq \frac{h}{2} \cdot (2g(\bar C)-2 +|Z|)=4h\chi_{\text{top}}(C)$;
\item[(iv)] suppose $N$ is an integer, coprime to $p$, such that $B_C[N]\rightarrow C$ is a split \'etale cover. Then the induced moduli map $C\rightarrow \mathscr{A}_{8h,1,N}$ extends to a map:

$$\bar{C}\rightarrow \mathscr{A}^*_{8h,1,N},$$
where the latter denotes the \emph{Baily-Borel} compactification. Then the Hodge line bundle $\alpha$ on $\mathscr{A}^*_{8h,1,N}$ pulls back to $\det(\Omega_{\bar{B}/\bar C})$.
\end{itemize}
\end{enumerate}
\end{lemma}
\begin{proof}
Construct $A_C\rightarrow C$ as in \autoref{corollary:isogenous_av}. Again, by Grothendieck's criterion for semistable reduction of abelian varieties, $A_C\rightarrow C$ must have semistable reduction. Set $B_C:=(A_C\times A_C^t)^4$; then by a result of Zarhin \cite[Chapitre IX, Lemme 1.1, p. 205]{moret1981familles}, $B_C\rightarrow C$ is principally polarized. Moreover, it clearly has semistable reduction. From the construction, and the fact that $A^t_C[p^{\infty}]\cong (A_C[p^{\infty}])^t$, where the first transpose is ``dual abelian scheme'' and the second is ``Serre-dual $p$-divisible group'', it follows that (2) holds.

We are left to prove (3). To do this, we will heavily use \cite[Setup A.10, Proposition A.11]{krishnamoorthypal2022}. First of all, each $\mathbb D(G_{C,i})$, a priori an Dieudonn\'e crystal on $C$, extends uniquely to logarithmic Dieudonn\'e crystal (with nilpotent residues) on $(\bar C,Z)$. Indeed, existence of the extension of $\mathbb D(B_C[p^{\infty}])$ follows from \cite[4.4-4.8]{katotrihan} and uniqueness from \cite[Proposition A.11(3)]{krishnamoorthypal2022}: name the extension $(M,F,V)$. These two results immediately imply the desired existence and uniqueness for the extension of $\mathbb D(G_i)$ to a logarithmic $F$-crystal, which we name $(M_i,F,V)$. The uniqueness implies that the (unique) extension of $\mathbb D(G_i^t)$ is isomorphic to the \emph{dual} logarithmic Dieudnn\'e crystal $(M_i,F,V)^t$ by \cite[5.11-5.12]{krishnamoorthypal2018}.

Set $M_{(\bar C,Z)}$ to be the \emph{evaluation} of $M$ on the trivial thickening of $(\bar C,Z)$ and set $\Omega$ to be the kernel of $F$ on $M_{(\bar C,Z)}$; then $\Omega$ is a vector bundle on $\bar C$, called the \emph{Hodge vector bundle}. (Kato-Trihan obtain the dual version of this in \cite[5.1]{katotrihan}, especially Lemma 5.3 of \emph{loc. cit.}.)\footnote{See also \cite[Remark A.8-A.9]{krishnamoorthypal2022} for a discussion of these results and references for the compatibility with the usual Hodge vector bundle associated to an abelian scheme.} Similarly, we can construct the Hodge bundle $\Omega_i$ of each $(M_i,F)$, which will be a \emph{line bundle} on $\bar{C}$. Moreover, there is a short exact sequence:

$$0\rightarrow \Omega_i\rightarrow (M_i)_{(\bar C,Z)}\rightarrow \Psi^*_i\rightarrow 0,$$ where $\Psi_i$ is the Hodge bundle of $G_i^t$.
We have an isomorphism of vector bundles on $\bar C$:

$$\displaystyle \Omega \cong \bigoplus_i (\Omega_i\times \Psi_i)^4$$

As each $G_{C,i}$ has non-empty supersingular locus, it follows that the \emph{Hasse invariant} associated to $G_{C,i}$: $$\text{Hasse}_{G_{C,i}}\in H^0(\bar{C},\Omega_i^{\otimes p-1})$$ is nonzero, which implies that $\Omega_i$ is a positive degree line bundle on $\bar{C}$. As $G_{C,i}^t$ is supersingular exactly when $G_{C,i}$ is supersingular, we deduce that $\Psi_i$ is also positive. Therefore, $\Omega=\Omega_{\bar{B}/\bar{C}}$ splits as the direct sum of $8h$ positive line bundles. In particular, we have shown (i). We further note that $\Omega$ is isomorphic to the Hodge line bundle associated to the N\'eron Model of $B_C\rightarrow C$ by \cite[A.11]{krishnamoorthypal2022} (this was first proven in \cite{katotrihan}).

For the next step: Faltings-Chai have constructed the following Kodaira-Spencer map \cite[Ch. III, Corollary 9.8]{faltings2013degeneration}:

\begin{equation}\label{equation:FC_KS}\Omega_{\bar{B}/\bar{C}}\otimes \Omega_{\bar{B}^t/\bar{C}}\rightarrow \omega_{\bar{C}}(Z),\end{equation}
extending the usual Kodaira-Spencer map over $C$. As $B$ admits a principal polarization, we have that $B_C\cong B^t_C$, and hence $\bar{B}_{\bar C}\cong \bar{B^t}_{\bar C}$, as both are simply the respective N\'eron models. Therefore, we may equivalently write \autoref{equation:FC_KS} as:

$$\theta\colon \Omega_{\bar{B}/\bar{C}}\rightarrow \Omega^*_{\bar{B}/\bar{C}}\otimes \omega_{\bar{C}}(Z)\cong \sHom(\Omega_{\bar{B}/\bar{C}},\omega_{\bar C}(Z)).$$

Under the decomposition:

$$\displaystyle\Omega_{\bar B/\bar C}\cong \bigoplus (\Omega_i\oplus \Psi_i)^4,$$
and after restricting to $C$, the above $\theta|_C$ is just the sum of the Kodaira-Spencer maps for each $G_i$ and $G_i^t$:

\begin{align*}
\Omega_i|_C\rightarrow &\sHom_{\mathcal O_C}((\Psi_i)|_C,\omega_{C})\\
\Psi_i|_C\rightarrow 	&\sHom_{\mathcal O_C}((\Omega_i)|_C,\omega_{C}).
\end{align*}
These were constructed to be nonzero, as both $G_i$ and $G_i^t$ are generically versally deformed; therefore, the Kodaira-Spencer map of sheaves is an injective map of coherent sheaves. Therefore (ii) is shown.

Fortunately, (iii) is an easy corollary of (ii). Indeed, taking degrees, we see that:

$$\deg(\Omega_{\bar B/\bar C})\leq \deg(\sHom(\Omega_{\bar{B}/\bar{C}},\omega_{\bar C}(Z)))=2g(\bar{C})-2+|Z|-\deg(\Omega_{\bar{B}/\bar{C}}),$$
from which the inequality follows immediately.

Finally, let us prove (iv). By moduli theory, we have a map $C\rightarrow \mathscr{A}_{8h,1,N}$, where the latter is a fine moduli scheme. As $\bar{C}$ is a smooth curve and the Baily-Borel compactification $\mathscr{A}^*_{8h,1,N}$ is proper, it follows that we get an extension:

$$\lambda\colon \bar{C}\rightarrow \mathscr{A}^*_{8h,1,N}.$$
Finally, the argument that the Hodge line bundle on $\mathscr{A}^*_{8h,1,N}$ pulls back under $\lambda$ to $\det(\Omega_{\bar B/\bar C})$ is given in the text surrounding \cite[Equations 3.4, 3.5]{krishnamoorthypal2022}. (While the argument in \emph{loc. cit.} is only written for $N=\ell$, it generalizes verbatim to the matter at hand. Indeed, the argument is an easy corollary of \cite[Ch. V, Theorem 2.5]{faltings2013degeneration}.)
\end{proof}
\section{The moduli spaces}\label{section:moduli}
We work in the following situation. Let $K$ be number field, let $N\geq 1$, set $S:=\Spec(\OKN)$, and let $\X/S$ be a smooth projective curve, let $\D\subset \X$ be a relative reduced divisor, and let $\U$ denote the open complement.

Let $\ell$ be a prime number and $g\geq 1$ an integer. We again denote by $\mathscr{A}^*_{g,1,\ell^3}$ the Baily-Borel compactification of $\mathscr{A}_{g,1,\ell^3}$, which is defined over $\Spec(\mathbb Z[1/\ell])$. This moduli space has a natural ample line bundle, the Hodge line bundle, which we denote by $\alpha$.

\begin{definition}\label{def_H} Fix a positive integer $b$. Denote by $\mathcal H$ the following contravariant pseudofunctor from the category of $S$ schemes to the 2-category of groupoids. The value $\mathcal H(T)$ for an $S$-scheme $T$ is the groupoid of triples $(\Y,\varphi,\lambda)$, that fit into a diagram as follows:

$$\xymatrix@1{
\Y\ar[rr]^-\lambda\ar[d]_{\varphi}	&& \mathscr A^*_{g,1,\ell^3}\times T\\
	\X_T,	&&
}$$
where
\begin{itemize}
\item $\Y/T$ is a smooth, projective, geometrically connected curve;
\item $\lambda$ sends $\W:=\varphi^{-1}(\U)\subset \Y$ to $\mathscr A_{g,1,\ell^3}\times T\subset \mathscr A_{g,1,\ell^3}^*\times T$;
\item $\varphi$ is finite morphism, of degree at most $\leq |\mathrm{GL}_{2g}(\bZ/\ell^3\bZ)|$, and \'etale over $\U$;
\item there exists some cusp $\infty$ of $\Y(T)$ (lying over a point of $\D(T)$ of $\X(T)$) that is sent to a 0-dimensional cusp of $\mathscr A^*_{g,1,\ell^3}\times T$;
\item and the degree of the pulled-back Hodge line bundle $\lambda^*(\alpha)$ on every geometric fiber of $\Y$ is $\leq b$,
\end{itemize}
with the natural notion of isomorphism: if $(\Y,\varphi,\lambda)$ and $(\Y',\varphi',\lambda')$ are elements of $\mathcal H(T)$, then an isomorphism between them is an $\X_T$-isomorphism $\Y\rightarrow \Y'$ that intertwines $\varphi$ and $\varphi'$ as well as $\lambda$ and $\lambda'$.
\end{definition}
Colloquially, the functor $\mathcal H$ parameterizes finite covers $\Y$ of $\X$ equipped with a (principally polarized) abelian scheme of dimension $g$ with trivial $\ell^3$ level structure, such that the induced map $\Y\rightarrow \mathscr{A}^*_{g,1,\ell^3}$ has ``degree'' bounded by $b$.

\begin{proposition}After potentially increasing $N$ (equivalently, replacing $S$ by a non-empty open subscheme), the functor $\mathcal H$ is represented by a finite type Deligne-Mumford stack over $S$.
\end{proposition}
\begin{proof}By increasing $N$, we can assume that all covers $\varphi\colon \Y\rightarrow \X$ that occur in the definition of $\mathcal H$ are tamely ramified at the cusps. More precisely, this will hold for any $N>|\text{GL}_{2g}(\mathbb Z/\ell^3\mathbb Z)|$.

It follows from the theory of the Hilbert scheme that for any noetherian scheme $T$ and for any relative smooth proper curve $\Y/T$ with geometrically connected fibers, the functor $\text{Hom}^{\leq b}_T(\Y,\mathscr A^*_{g,1,\ell^3})$ parametrizing maps $\lambda$ such that the $\deg(\lambda^*(\alpha))\leq b$ for every geometric fiber is represented by a finite-type $T$-scheme. (In particular, this holds true even if $T$ is not connected and the genus of the fibers varies on different connected components.) It follows that if $T$ is a Deligne-Mumford stack, of finite type over $S$, and $\Y/T$ is a smooth, proper curve with geometrically connected fibers, then the same functor $\text{Hom}^{\leq b}_T(\Y,\mathscr A^*_{g,1,\ell^3})$ is represented by a finite type Deligne-Mumford stack over $T$.

On the other hand, the theory of the Hurwitz scheme implies that the functor $\text{Cov}^{c}_{(\X,\D)/S}$ parametrizing finite tame covers $\Y\rightarrow \X$ of degree $\leq c$ such that:
\begin{itemize}
\item $\Y/S$ has geometrically connected fibers and
\item $\Y\rightarrow \X$ is \'etale over $\U:=\X\setminus \D$
\end{itemize}
is represented by a finite-type Deligne-Mumford stack over $S$.

There is a natural map $$\displaystyle \text{Cov}^{\leq c}_{(\X,\D)/S}\rightarrow \mathscr M:= \bigsqcup _{k\text{ bounded}} \mathscr M_{k},$$
which is the map that sends a cover $\Y\rightarrow \X$ to the underlying curve $\Y$. Here, the notation $\mathscr M_{k}$ stands for the moduli space of genus $k$ curves. Denote by $\mathscr C_k\rightarrow \mathscr M_k$ the universal curve and by $\mathscr C:=\bigsqcup \mathscr C_k$, which has a natural map $\mathscr C\rightarrow \mathscr M$.

Then consider $\mathscr H$, the open substack of the 2-fiber product:
$$\text{Hom}^{\leq b}_{\mathscr M}(\mathscr C,\mathscr A^*_{g,1,\ell^3}))\times_{\mathscr M} \text{Cov}^{\leq |\text{GL}_{2g}(\mathbb Z/\ell^3\mathbb Z)|}_{(\X,\D)/S},$$
which corresponds to the condition that sends $\lambda(\varphi^{-1}(\U))\subset \mathscr A_{g,1,\ell^3}$, i.e., that $\varphi^{-1}(\U)$ is sent inside of the moduli space of abelian varieties. It follows that $\mathscr H$ is finite type Deligne-Mumford stack over $S$. By further imposing the condition that the map $\lambda$ sends at least one point in the boundary divisor to a zero dimensional cusp of the Baily-Borel compactification, $\mathcal H$ is a closed substack of $\mathscr{H}$, which is again finite type.
\end{proof}

Now, let $\bL$ be a lisse $\Qlbar$-sheaf as in \autoref{theorem:main}. There exists an $\ell$-adic local field $M/\Ql$ such that the associated representation factors through the ring of integers $\mathcal O_M$:
$$\rho\colon \pi_1(\X_{\OKN})\rightarrow \text{GL}_2(\mathcal O_M)\subset \text{GL}_2(\Qlbar).$$
Abusing notation, we call the induced {lisse $\mathcal O_M$-sheaf} $\mathcal L$. Denote by $\pi_M$ the uniformizer of $M$ and $\kappa_M$ the residue field of $M$.

 \begin{definition}
 Fix $i\geq 1$ and a lattice $\mathcal L$ as above. Let $\tilde{\mathcal{H}}_{i}$ denote the following contravariant pseudofunctor from $S$-schemes to groupoids: the value $\tilde{\mathcal H}_i(T)$ on an $S$-scheme $T$ is the collection of quadruples $(\Y,\varphi,\lambda, \psi)$, where $(\Y,\varphi,\lambda)\in \mathcal H(T)$, and $\psi$ is the following extra piece of data. As $\lambda\colon \W:=\varphi^{-1}(\U)\rightarrow \mathscr{A}_{g,1,\ell^3}$, there is a principally polarized abelian scheme $f\colon A_{\W}\rightarrow \W$ (with trivial $\ell^3$-torsion). Then
 	$$\psi\colon \varphi^{*}(\mathcal L/\pi_M^i)\rightarrow R^1f_*\mathcal O_M/\pi_M^i$$
 is a map of \'etale torsion sheaves on $\W$ whose reduction modulo $\pi_M$-reduction is nonzero. In other words, $\mathrm{im}(\psi)\not\subseteq \pi_M(R^1f_*\mathcal O_M/\pi_M^i)$. There is an obvious notion of isomorphism of two such quadruples.
 \end{definition}

 The pseudo-functor $\tilde{\mathcal H}_i$ is actually a stack in the \'etale topology. This follows from the following two properties. Let $T$ be a scheme:
 \begin{itemize}
 \item There exists an internal Hom in the category of torsion locally constant abelian \'etale sheaves on $T$.
 \item If $\psi\colon \mathcal F\rightarrow \mathcal G$ is a map of torsion, locally constant \'etale sheaves of $\mathcal O_M$ modules on $T$, then the property that $\psi(\mathcal F)\not\subset \pi_M(\mathcal G)$ may be checked on an \'etale cover.
 \end{itemize}
 There are natural transformations of pseudo-functors $\tilde{\mathcal H}_{j}\rightarrow \tilde{\mathcal H}_{i}$ for any $j>i$. We claim that $\tilde{\mathcal H}_i$ represents a finite type Deligne-Mumford stack over $S$. To prove this, it suffices to prove that the natural transformation of pseudo-functors $\tilde{\mathcal H}_i\rightarrow \mathcal H$ is representable by a scheme.

 Let $T$ be an $S$-scheme, and $t:=(\Y,\varphi,\lambda)\in \mathcal H(T)$. Then we have the following pullback square:

 $$\xymatrix{
 T\times_{\mathcal H} \tilde{\mathcal H}_i\ar[r]\ar[d]	&	\tilde{\mathcal H}_i\ar[d]
 \\
 T \ar[r]^t		& \mathcal{H}
 }$$
 Then $T\times_{\mathcal H}\tilde{\mathcal H}_i$ has the following description. There are two natural $\ell^i$-torsion \'etale sheaves on $\Y$: $\varphi^*(\mathcal L/\ell^i)$ (which has $\mathcal O_M/\ell^i$-rank 2), and $R^1f_*\mathcal O_M/\ell^i$ (which has $\mathcal O_M/\ell^i$-rank $2g$). Then $T\times_{\mathcal H}\tilde{\mathcal H}_i$ corresponds to the (finite) set of injective maps of sheaves of abelian groups: $\psi\colon \varphi^{*}(\mathcal L/\ell^i)\hookrightarrow R^1f_*\mathcal O_M/\ell^i$. This finite set is canonically a scheme. It follows that $\tilde{\mathcal H}_i\rightarrow \mathcal H$ is relatively representable, and hence $\tilde{\mathcal H}_i$ is represented by a Deligne-Mumford stack of finite type over $S$.
 \begin{proposition}\label{prop:psi_finite}
 The natural map ``forget $\psi$'': $\tilde{\mathcal H_i}\rightarrow \mathcal H$ is finite.
 \end{proposition}
 \begin{proof}
 It is obviously quasi-finite because, as argued above, if we fix $i$, then there are only finitely many choices for $\psi$. To prove it is finite, we show that it is proper.

 As both $\tilde{\mathcal H}_i$ and $\mathcal H$ are of finite type over $S=\Spec(\OKN)$, it suffices to simply check the valuative criterion for properness. Let $R$ be a discrete valuation ring with fraction field $F$. Suppose we have $(\Y,\varphi,\lambda) \in \mathcal H(R)$ and $(\Y_F,\varphi_F,\lambda_F,\psi_F)\in \tilde{\mathcal H}_i(F)$. Therefore we have a principally polarized abelian scheme $f'\colon A_{\W}\rightarrow \W$ (of dimension $g$, with trivial $\ell^3$ torsion), together with a map of torsion \'etale sheaves over $\W_F$ whose reduction modulo $\pi_M$ is nontrivial:

 $$\psi_F\colon \varphi^*(\mathcal L)/\pi_M^i|_{\W_F}\rightarrow R^1f'_{*}\mathcal O_M/\pi_M^i|_{\W_F} .$$
Note the following. If one has two finite \'etale sheaves on an irreducible normal scheme, and a morphism between them over the generic point, then that morphism uniquely extends to the whole scheme. (Here, we are closely following \cite[Proof of Lemma 23]{snowden2018constructing}.) Therefore, $\psi_F$ extends to a $\psi$ on all of $\U'_R$, and we have verified the valuative criterion for properness.
\end{proof}

\begin{definition}\label{def:H_infinity}For $i\geq 1$, set $$\mathcal H_i:=\text{Im}(\tilde{\mathcal H_i}\rightarrow \mathcal H).$$
 As $\tilde{\mathcal H}_i\rightarrow \mathcal H$ is relatively representable and finite (\autoref{prop:psi_finite}), it is universally closed. Therefore $\mathcal H_i$ is a closed subset of $|\mathcal H|$, which we may equip with the induced reduced substack structure \cite[Tag 0508]{stacks-project}. According the natural transformations $\tilde{\mathcal H}_{j}\rightarrow \tilde{\mathcal H}_{i}$, the sequence of closed subsets are descending. Set:
 $$\mathcal H_{\infty}:={\bigcap_{j}\mathcal H_j},$$
 which is also equipped with the reduced induced substack structure. Then $\mathcal H_i$ and $\mathcal H_{\infty}$ are Deligne-Mumford stacks of finite type over $S=\Spec(\OKN)$ for all $i\geq 1$.
 \end{definition}

 \begin{lemma}\label{remark:Hk}Let $T$ be an $S$-scheme and let $(\Y,\varphi,\lambda)\in \mathcal H(T)$. Then the following conditions are equivalent:
 \begin{enumerate}
 	\item $(\Y,\varphi,\lambda)\in \mathcal H_\infty(T)$;
 	\item there exists an injection
 	$$\varphi^*(\mathcal L)\hookrightarrow R^1f_*\mathcal O_M$$
 	of lisse $\mathcal O_M$-sheaves on $T$;
 	\item there is an injection $\tau\colon \varphi^*(\mathcal L\otimes M)\hookrightarrow R^1f_*M$ of lisse $M$-sheaves on $T$.
 \end{enumerate}
 \end{lemma}

\begin{proof}
\textbf{(2) $\Rightarrow$ (3):} By applying $- \otimes_{\mathcal O_M} M$ to the injection $\varphi^*(\mathcal L)\hookrightarrow R^1f_*\mathcal O_M$ , we get the desire injection. (Note that both are lisse $\mathcal O_M$-sheaves, so tensoring with $M$ yields an injective map.)

\textbf{(3) $\Rightarrow$ (2):} Since $R^1f_*\mathcal O_M$ (resp. $\mathcal L$) is an $\mathcal O_M$-lattice in $R^1f_*M$ (resp. $\mathcal L\otimes M$), there exists some integer $\iota$ such that
\[\pi_M^\iota\cdot \tau\left(\varphi^*(\mathcal L)\right)\subset R^1f_*\mathcal O_M.\]
Then the map $\pi_M^\iota\tau$ is an injection from $\varphi^*(\mathcal L)$ to $R^1f_*\mathcal O_M$.

\textbf{(2) $\Rightarrow$ (1):} Denote by $\psi'$ the injection in (2). It is clear there exists some integer $\lambda$ such that
\[\mathrm{im}(\psi')\subset \pi_M^\lambda R^1f_*\mathcal O_M \quad \text{and} \quad \mathrm{im}(\psi')\not\subset \pi_M^{\lambda+1} R^1f_*\mathcal O_M.\]
Denote $\psi = \frac{\psi'}{\pi_M^\lambda}$ which is clearly an injection from $\varphi^*(\mathcal L)$ to $R^1f_*\mathcal O_M$ and satisfies
\[\mathrm{im}(\psi)\not\subset \pi_M R^1f_*\mathcal O_M.\]
This is equivalent to saying that the reduction modulo $\pi_M$ of $\psi$ is nontrivial.
Denote $\psi_i = \psi\mod(\pi_M^i)$ for each $i>0$,
\[\psi_i \colon \varphi^*(\mathcal L)/\pi^i_M \rightarrow R^1f_*\mathcal O_M/\pi^i_M.\]
Since $\psi_i\mod(\pi_M) = \psi\mod(\pi_M) \neq 0$, the quadruple $(\Y,\varphi,\lambda,\psi_i)\in \tilde{\mathcal H}_i$. Thus $(\Y,\varphi,\lambda)\in \cap_{i=1}^\infty \mathcal H_i = \mathcal H_\infty$.

\textbf{(1) $\Rightarrow$ (3):} (This is the main content of the Lemma.) Since $(\Y,\varphi,\lambda)\in \mathcal H_\infty(T)$, for each $i>0$, there exists a map
\[\psi'_i \colon \varphi^*(\mathcal L)/\pi^i_M \rightarrow R^1f_*\mathcal O_M/\pi^i_M\]
which is nontrivial modulo $\pi_M$. In general, the $\psi'_i$'s \textbf{do not} form a compatible sequence, i.e., it is possible that there exists $j>i$ with the following property: $\psi'_{j} \mod{(\pi_M^i)} \not\equiv \psi'_i$. Therefore, one can not directly take projective limits to find our desired map $ \varphi^*(\mathcal L) \rightarrow R^1f_*\mathcal O_M$. However, we claim we may derive a compatible sequence from $\psi_i'$ as follows.

Consider the subset in the finite set
\[\Sigma_1 = \mathrm{Hom}(\varphi^*(\mathcal L)/\pi_M, R^1f_*\mathcal O_M/\pi_M)\]
consisting of all modulo $\pi_M$ reductions of $\psi_i'$:
\[\{\psi'_i\mod(\pi_M) \mid i\geq 1\}.\]
By the pigeonhole principle, there exists a nontrivial map $\psi_1\in \Sigma_1$ and an infinite subset $\mathbb N_1\subset \mathbb N$ such that $\psi'_i\mod(\pi_M) = \psi_1$ for any $i\in\mathbb N_1$.

Suppose we have constructed a compatible sequence $\psi_1,\psi_2,\cdots,\psi_r$ and an infinite subset $\mathbb N_r\subset \mathbb N$ satisfying
\[\psi'_i\mod(\pi_M^j) = \psi_j \in \mathrm{Hom}(\varphi^*(\mathcal L)/\pi_M^j, R^1f_*\mathcal O_M/\pi_M^j) \]
for any $i\in\mathbb N_r$ and $j\in\{1,2,\cdots,r\}$. Then we consider the subset in the finite set
\[\Sigma_{r+1} = \{\rho\in \mathrm{Hom}(\varphi^*(\mathcal L)/\pi_M^{r+1}, R^1f_*\mathcal O_M/\pi_M^{r+1}) \mid \rho\mod{\pi_M^r} = \psi_r\}\]
consisting of all modulo $\pi_M^{r+1}$ reductions
$\psi'_i\mod(\pi_M^{r+1})$:
\[\{\psi_i' \mod (\pi_M^{r+1})\mid i\in\mathbb N_r\}.\]
Again by the pigeonhole principle, there exists a nontrivial map $\psi_{r+1}\in \Sigma_{r+1}$ and an infinite subset $\mathbb N_{r+1}\subset \mathbb N_r$ such that $\psi'_i\mod(\pi_M^{r+1}) = \psi_{r+1}$ for any $i\in\mathbb N_{r+1}$.

Iteratively, we find a sequence $\psi_1,\psi_2,\cdots$ satisfying
\[\psi_j \mod(\pi_M^i) = \psi_i\]
for each $j>i$. Taking projective limits and tensoring with $M$, one gets a nonzero map
\[\psi\colon \varphi^*(\mathcal L\otimes M)\rightarrow R^1f_* M.\]
Since $\mathcal L\otimes M$ is irreducible, $\psi$ is injective.
\end{proof}

 \section{Rigidity}\label{section:rigidity}
In this section, we prove the following. Recall that $S=\Spec(\OKN)$
 \begin{lemma}Let $\mathcal H/S$ be as in \autoref{section:moduli}. Then, after potentially increasing $N$ (equivalently, replacing $S$ by a non-empty Zariski open subset), the relative dimension of $\mathcal H/S$ is 0.
 \end{lemma}
 \begin{proof}We have shown that $\mathcal H/S$ is a finite type Deligne-Mumford stack. To show the desired result, it suffices to show that that if $K\hookrightarrow \mathbb C$ is an embedding, then $\mathcal H_{\CC}$ has dimension 0. Equivalently, we want to show that if $A_{U_{\CC}}\rightarrow U_{\CC}$ is a principally polarized abelian scheme that is totally degenerate at at least one cusp, then it is rigid. This immediately follows from Theorem 8.6 together with Lemma 3.4 and the following text of \cite{saitorigid}.
 \end{proof}

\section{The proof}
\begin{proof}[Proof of \autoref{theorem:main}]
First, assume that $\bL$ has bad, unipotent reduction around every cusp. Let $\mathcal T_1$ be the set of those prime $\p$ of $\mathcal O_K$ with the following properties: the underlying prime $p$ splits completely in $E$, and $p>\max(N,\ell^3)$. This is an infinite set by the Cebotarev density theorem. Let $\mL_{\p}$ be the restriction of $\mL$ to $U_{\p}$. Then $\mL_{\p}$ is irreducible by exactly the same argument as that of the first paragraph of \cite[Proof of Lemma 24, p. 2053]{snowden2018constructing}.

There are only finitely may subfields of $E$. It follows from the pigeonhole principle that there exists a subfield $F\subset E$ such that there exists infinitely many primes $\p\in \mathcal T_1$ such that $\mL_{\p}$ has Frobenius traces in $F\subset E$. Call the collection of such primes $\mathcal T_2\subset \mathcal T_1$.
Let $\mathcal H$ and $\mathcal H_{\infty}$ be the moduli spaces from \autoref{section:moduli} with $g=8[F:\bQ]$ and $b=4h\chi_{\text{top}}(U)$. Note that, after increasing $N$, both spaces have relative dimension $0$ over $\OKN$ by Section \ref{section:rigidity}.

First of all, note that for each $\p\in \mathcal T_2$, $\mathcal H_{\infty}(\kappa(\p))\neq \emptyset$. This follows by \autoref{lemma:main_char_p}, especially part (3.iv), together with \autoref{remark:Hk}. In more detail: \autoref{lemma:main_char_p} implies that we can find an abelian scheme $B_{\U_{\p}}\rightarrow {\U_{\p}}$ such that $\mathcal L_{\p}^4$ injects in the cohomology, that has semistable reduction at infinity, and such that the Hodge bundle on $\X_{\p}$ has bounded degree. \autoref{remark:Hk} then implies that such an abelian scheme corresponds to a point $\beta_{\p}$ in $\mathcal H_{\infty}(\kappa(\p))\subset \mathcal H(\kappa(\p))$. Since $\mathcal T_2$ is infinite and $\mathcal H_{\infty}/S$ is of finite type, it follows that there exists a finite field extension $K'/K$ and a point $\beta\in \mathcal H(K')$. In fact, as $\mathcal H_{\infty}$ has relative dimension $0$, our point $\beta$ may be chosen to be compatible with infinitely many of the $\beta_{\p}$, where compatibility is defined in the obvious sense. By definition of $\mathcal H_\infty$, the point $\beta \in \mathcal H_\infty(K')$ corresponds to an abelian scheme $B_{U'_{K'}}\rightarrow U'_{K'}$ such that $\mathcal L|_{U'_{K'}}$ injects into the integral $\mathcal O_M$ cohomology. By taking a Weil restriction, we obtain an abelian scheme $A_U\rightarrow U$ (of dimension $g[K':K]$) such that $\bL$ injects into the the cohomology of $A_U\rightarrow U$. Using Faltings' semi-simplicity theorem, we conclude that $\bL$ is in fact a summand of the cohomology, as desired.

In general, there exists a finite \'etale cover $f\colon U'\rightarrow U$ such that $f^*\mathbb L$ has the following property. Let $C'$ be the compactification of $U'$, and set $D'$ to be the divisor at infinity. Then for each $\infty\in D'$, the lisse $\ell$-adic sheaf $f^*\mathbb L$ has either good reduction at $\infty$ or bad, unipotent reduction at $\infty$. There then exists a curve $U'\subset V'\subset C'$, where $f^*\mathbb L$ extends to a lisse $\ell$-adic sheaf $\mathbb M'$ on all of $V'$, and moreover, has bad, unipotent reduction around every point in $C'\setminus V'$. Then the above argument applies, producing an abelian scheme $A_{V'}\rightarrow V'$ whose cohomology has $\mathbb M'$ as a summand. Restricting to $U'$ and then applying a Weil restriction of scalars along the finite \'etale map $U'/U$, we obtain the desired result.
\end{proof}

\appendix
\section{Frobenius untwisting and a second proof of \autoref{prop:vers_def}(2)}
In this appendix, we have two goals: we first provide a proof of \cite[Theorem 6.1]{xia2013deformation} in the context we need, which we use several times, and then we provide a second perspective on the termination of the Frobenius untwisting process in the context of \autoref{prop:vers_def}.

Before we begin the proofs, we need one preliminary claim. Let $(\bar{\mathcal C}, \mathcal{Z})$ be a lift of $(\bar C,Z)$ over $W(k)$. Let $(\mV,\nabla)$ be a vector bundle together with a logarithmic connection with nilpotent residues on $(\bar{\mathcal C},\mathcal Z)/W(k)$, such that $\nabla$ is topologically quasi-nilpotent. (Therefore, $(\mV,\nabla)$ is the \emph{value} of a logarithmic crystal $(C,Z)$ on the particular thickening $(\bar{\mathcal C},\mathcal{Z})$).

\begin{claim}The following two properties hold for $(\mV,\nabla)$.
\begin{enumerate}
\item[(i)] The logarithmic isocrystal $(\mV,\nabla)\otimes \Qp$ over $(\mathcal{\bar{C}},\mathcal{Z})|_{W(k)[1/p]}$ is semistable and of degree 0.
\item[(ii)] The degree of $\mV_p=\mV\otimes\mathbb F_p$, the restriction of the vector bundle $\mV$ to $\bar C$, is 0.
\end{enumerate}
\end{claim}
\begin{proof}[Proof of claim]
To prove that $\mV\otimes \Qp$ has degree 0, it suffices to base change along a map $W(k)[1/p]\hookrightarrow \mathbb C$. Then the result follows from a computation of as Esnault-Viehweg \cite[Appendix B]{esnaultviehweg}. Now semistability follows easily. Indeed, any horizontal subsheaf of $\mV\otimes \Qp$ is necessarily a bundle, which is therefore equipped with a logarithmic flat connection and has nilpotent residues. By the first sentence, this implies that this horizontal subsheaf has degree 0, validating semistability.

To prove the second statement, it suffices to note that degree, being the first Chern class, is locally constant, see \cite[Section 6]{KYZ20A}; therefore, if the degree of $\mV\otimes \Qp$ is $0$ on $\bar{\mathcal C}$, then so is the degree of $\mV_{p}$ on $\bar C$.
\end{proof}
The third term in the following lemma is a special case of \cite[Theorem 6.1]{xia2013deformation} that we need. In particular, we work in the context of strictly semistable $p$-divisible groups on $(\bar{C},Z)$, as this allows us to discuss the destabilizing iteration.

\begin{lemma}\label{Xia}
Let $G_C\rightarrow C$ be a strictly semistable height $2$, dimension $1$ $p$-divisible group on $C$. Suppose the Kodaira-Spencer map of $G_C\rightarrow C$ is $0$.
\begin{enumerate}
	\item Let $\mathbb D(G_C)$ be the Dieudonn\'e module of $G_C$. Then the Dieudonn\'e crystal $\mathbb D(G_C)$ canonically extends to a logarithmic Dieudonn\'e crystal on $(\bar{C},Z)$.
	\item Set $(\mM,\nabla, F,V)$ denote the evaluation of the logarithmic extension of $\mathbb D(G_C)$ on the log pair $(\bar{\mathcal{C}}, \mathcal Z)$. Then the Hodge line bundle $L$ in $\mathcal M_p = \mathcal M\otimes \mathbb F_p$ has positive degree, which is the maximal destabilizing subbundle	of $(\mM_p,\nabla_p)$.
	\item Then there exists an isogenous $p$-divisible group $G'_C\rightarrow C$, such that the Frobenius pullback $G'_C{}^{(p)}$ is isomorphic to $G_C$.
\end{enumerate}
\end{lemma}

\begin{proof}
 Since $G_C$ is semistable, the first term follows from \cite[Corollary 3.14]{trihanBT}. And the second term follows from the existence of supersingular points as in \autoref{prop:vers_def}(1) via the Hasse-Witt map. Consider the Kodaira-Spencer map
$$\theta: L\to \mM_p/L\otimes \Omega^1_{\bar C}(\log Z).$$
By assumption $\theta=0$, thus $L\subset (\mM_p,\nabla_p)$ is a horizontal subbundle. Since $\deg \mM_p=0$ and $\deg L>0$, the line bundle $L$ is just the maximal destabilizing subbundle of $(\mM_p,\nabla_p)$.

For the third term, We mainly follows Xia's original proof.
Set $(\mM',\nabla')$ to be the kernel of the following composition map
\[ (\mM,\nabla) \xrightarrow{\quad\pi\quad} (\mM_p,\nabla_p)\to (\mM_p,\nabla_p)/(L,\nabla_p)=:(N,\nabla_p),\]
where $\pi\colon (\mM,\nabla)\rightarrow (\mM_p,\nabla_p)$ is the reduction modulo $p$ map. In particular, one has
\begin{equation}
	p \mM\subset \mM', \quad \pi(\mM') = \mM'/p \mM = L, \quad \text{and} \quad N = \mM/\mM'.
\end{equation}

The crucial point is to show the Frobenius structure and the Verschiebung extend; if we show this, then we will obtain a new logarithmic Dieudonn\'e module $(\mM',\nabla',F',V').$

Locally, over an affine open subset $\mathcal U=\mathrm{Spec}(R)$, we choose a lifting $\Phi\colon \widehat{R}\rightarrow \widehat{R}$ of the absolute Frobenius map $\sigma\colon R/pR\rightarrow R/pR$. Then the Frobenius structure and Verschiebung structure are given by
\[F\colon \mM(U)^\Phi \rightarrow \mM(U) \quad \text{and} \quad V\colon \mM(U) \rightarrow \mM(U)^\Phi.\]
Recall \cite[Proposition 2.5.2]{de1995crystalline}, $L$ is the unique subbundle of $\mM_p$ such that
\begin{equation} \label{subbundle}
	(L)^\sigma = \mathrm{im}(V_p) = \ker(F_p).
\end{equation} In particular,
\begin{equation*}
	\pi\big(F(\mM'(U)^\Phi)\big) = F_\p\big( \pi( \mM'(U)^\Phi)\big) = F_\p\big( \pi( \mM'(U))^\sigma\big) \overset{(\ref{BasicFactors})}{=} F_p((L(U))^\sigma) \overset{(\ref{subbundle})}{=} 0.
\end{equation*}
This implies that
\[F(\mM'(U)^\Phi) \subseteq \p \mM(U) \subseteq \mM'(U)\]
and that the Frobenius structure $F$ can be restricted onto $(\mM',\nabla')$, denoted by $F'$. Similarly,
\begin{equation}
	\pi\big(V(\mM'(U))\big) \subseteq \pi \big( V(\mM(U)) \big)= \mathrm{im}(V_p)(U) \overset{(\ref{subbundle})}{=} (L)^\sigma
\end{equation}
This implies that
\[	V(\mM'(U)) \subseteq \pi^{-1}((L)^\sigma ) \overset{(\ref{BasicFactors})}{=} \mM'(U)^\Phi.\]
and that the Verschiebung structure $V$ can be restricted onto $(\mM',\nabla')$, denoted by $V'$. The module $(\mM, \nabla, F, V)'$ is the realization of the $\mathbb D(G'_C)$ of a $p$-divisible group $G'_C$ which satisfies $G'_C{}^{(p)}=G$ by de Jong's fundamental theorem.
\end{proof}

\begin{proof}[Second proof of \autoref{prop:vers_def} (2)]
 Recall that a $p$-divisible group is called \emph{generically versally deformed} if the corresponding Kodaira-Spencer map is nonzero. From the third term in \autoref{Xia}, one may construct inductively
 \begin{enumerate}
 	\item an infinite sequence of $p$-divisible groups over $C$
 	\[G^0_C=G_C,G^1_C,G^2_C,\cdots\]
 	such that $(G^{i+1}_{C})^{(p)}=G^i_C$ for all $i\geq0$ and whose Kodaira-Spencer maps are all zero; or
 	\item a finite sequence of $p$-divisible groups
 	\[G^0_C=G_C,G^1_C,\cdots, G^r_C\]
 	such that $(G^{i+1}_{C})^{(p)}=G^i_C$, the Kodaira-Spencer maps of $G^i_C$ are zero for all $i\in \{0,\cdots,r-1\}$, and the Kodaira-Spencer map of $G^{r}_C$ is nonzero;
 \end{enumerate}
To prove the second term of \autoref{prop:vers_def}, one only need to show that the first case does not appear. Suppose we are in the first case. By a method of Langer in \cite[Theorem 5.1]{Lan14}, we will construct a contradiction.

Let $(\mM^i,\nabla^i,F^i,V^i)$ be the logarithmic Dieudonn\'e module associated to $G^i_C$ and let $L^i$ be the Hodge line bundle in $\mM_p^i$. According the construction of $G^i_C$ as in the proof of \autoref{Xia}, one has
	\[ (\mM^0,\nabla^0, F^0, V^0) \supsetneq (\mM^1,\nabla^1,F^1,V^1)\supsetneq (\mM^2,\nabla^2, F^2, V^2) \supsetneq \cdots\]
where $(\mM^{i+1},\nabla^{i+1})$ to be the kernel of the following composition map
	\[ (\mM^i,\nabla^i) \xrightarrow{\quad\pi\quad} (\mM^i_p,\nabla^i_p)\to (\mM^i_p,\nabla^i_p)/(L^i,\nabla^i_p)=:(N^i,\nabla^i_p),\]
where $\pi\colon (\mM^i,\nabla^i)\rightarrow (\mM^i_p,\nabla^i_p)$ is the reduction modulo $p$ map. In particular, one has
\begin{equation} \label{BasicFactors}
	p \mM^i\subset \mM^{i+1}, \quad \pi(\mM^{i+1}) = \mM^{i+1}/p \mM^i = L^i, \quad \text{and} \quad N^i = \mM^i/\mM^{i+1}.
\end{equation}
Denote
	\[(\overline{\mM}^m,\overline{\nabla}^m,\overline{F}^m,\overline{V}^m):= \varprojlim_{n\geq m} (\mM,\nabla,F,V)^m /(M,\nabla,F,V)^n.\]
In the following, we show that the generic fiber $\overline{\mM}^m \otimes K$ of $\overline{\mM}^m$ is a destabilizing quotient of $\mM^m\otimes K = \mM_K$ for sufficiently large $m\gg0$; this will contradict semi-stability of $\mM_K$.

From \autoref{BasicFactors} and the construction of the sequence, one has exact sequences of $\mathcal O_{\X_\p}$-modules
	\[0\rightarrow L^n \rightarrow \mM^n_p \rightarrow N^n \rightarrow 0\]
	and
	\[0 \rightarrow N^n \xrightarrow{p} \mM^{n+1}_p \rightarrow L^n \rightarrow 0\]
	Denote by $C^n$ the kernel of the composition $L^{n+1} \rightarrow M_p^{n+1} \rightarrow L^n$. By the same reason as in the proof of \cite[Theorem 5.1]{Lan14}, $C^n=0$ for sufficient large $n$. By eliminating the first finitely many terms, we may assume $C^n=0$ for all $n\geq0$. Thus one gets injections
	\[L^0\supseteq L^1\supseteq L^2\supseteq \cdots \quad \text{and} \quad N^0\subseteq N^1\subseteq N^2 \subseteq \cdots\]
	Since the slope of $L^n$ is non-increasing and also non-negative, the sequence $L^n$ stabilizes to some $L$. Since $\deg L^n + \deg N^n = \deg \mM^n_\p = 0$, the sequence $N^n$ stabilizes to some $N$. Once again, by eliminating the first finitely many terms, we may assume $L^n = L$ and $N^n=N$ for all $n\geq0$. In particular, the composition map \[\mM^n/\mM^{n+1}=N^n \xrightarrow{p} \mM^{n+1}_\p\rightarrow N^{n+1}=\mM^{n+1}/\mM^{n+2}\]
	is an isomorphism for all $n\geq0$. Thus
	\begin{equation} \label{equ:aa11}
		\mM^{n+1} = p \mM^n + \mM^{n+2}
	\end{equation}
and
\begin{equation} \label{equ:aa22}
	 p \mM^{n+1} = p \mM^n \cap \mM^{n+2}
\end{equation}
for each $n\geq0$. We show that $\mM^i=p^{i}\mM^0 + \mM^n$ for any $0\leq i\leq n$ and $p \mM^{n-1} = p \mM^0 \cap \mM^n$ as follows:
	\begin{equation} \label{equ:aa33}
		\begin{split}
			\mM^i & \overset{\eqref{equ:aa11}}{=} p\mM^{i-1} +\mM^{i+1} \overset{\eqref{equ:aa11}}{=} p\mM^{i-1} + p\mM^{i+1} + \mM^{i+2} \\ & \overset{\eqref{equ:aa11}}{=}p\mM^{i-1} + p\mM^{i+1} + \cdots + p\mM^{n-1} + \mM^n = p\mM^{i-1} + \mM^n,\\
		\end{split}
	\end{equation}
\begin{equation}
	\begin{split}
		\mM^i & \overset{\eqref{equ:aa33}}{=} p\mM^{i-1} + \mM^n = \underbrace{p(p(\cdots p(p}_i\mM^0+\mM^n)+ \mM^n)\cdots + \mM^n) + \mM^n\\
		& = p^i\mM^0 + \mM^n,\\
	\end{split}
\end{equation}
and
\begin{equation}
	\begin{split}
		p\mM^{n-1} & \overset{\eqref{equ:aa22}}{=} p\mM^{n-2} \cap \mM^n \overset{\eqref{equ:aa22}}{=} (p\mM^{n-3} \cap \mM_{n-1} )\cap \mM^n \\
		& \overset{\eqref{equ:aa22}}{=} ((\cdots(p\mM^{0}\cap \mM_1)\cap \cdots) \cap \mM_{n-1} )\cap \mM^n = p\mM^0\cap \mM^n\\
	\end{split}
\end{equation}
In particular, there is an $W_n=W(k)/p^n$ module structure on $\mM^0/\mM^n$. Consider
	\begin{equation*}
		\begin{split}
			(\mM^0/\mM^n)\otimes_{W_n} (pW_n/p^nW_n) & \cong (\mM^0/\mM^n)\otimes_{W_n} W_n/p^{n-1}\\
			& = \mM^0/(p^{n-1} \mM^0+\mM^n) = \mM^0/\mM^{n-1} \\
			& \cong p \mM^0/ p \mM^{n-1} = p \mM^0/ (p \mM^0 \cap \mM^n)\\
			& \cong (p \mM^0+\mM^n)/\mM^n = \mM^1/\mM^n \\
		\end{split}
	\end{equation*}
	where the map can be easily checked to be given by $\overline{m} \otimes \overline{a} \mapsto \overline{am}$ for any $m\in \mM^0$ and any $a\in p$. One also has
	\[(\mM^0/\mM^n)\otimes_{W_n} k = \mM^0/(p \mM^0 + \mM^n) = \mM^0/\mM^1\]
	which implies the following sequence is exact
	\[0\rightarrow (\mM^0/\mM^n)\otimes_{W(k)} pW_n/p^n\rightarrow (\mM^0/\mM^n) \otimes_{W} W(k)/p^n \rightarrow (\mM^0/\mM^n)\otimes_{\mathcal O_K} k\rightarrow 0.\]
	Thus $\mathrm{Tor}^{W/p^n}(k,\mM^0/\mM^n)=0$ and $\mM^0/\mM^n$ is flat over $W_n$.
	It follows that $\widetilde{\mM}^0 = \varprojlim_{n\geq0} \mM^0/\mM^n$ is an $W(k)$-flat coherent $\mathcal O_{\bar{\mathcal C}}$-module having a filtration with quotients isomorphic to $N$. Thus $\widetilde{\mM}^0\otimes W(k)[1/p]$ is a destabilizing quotient of $\mM^0\otimes W(k)[1/p],$	 which contradicts the semistability of $\mM^0\otimes W(k)[1/p].$
\end{proof}

\begin{acknowledgement*} We would like to express our sincere gratitude to the referees for their valuable feedback and constructive criticism, which greatly contributed to the improvement of this paper.
R.K. warmly thanks Xia Jie for discussions about his thesis in 2013, Johan de Jong for explaining the importance of the condition 'generically versally deformed' for height $2$, dimension $1$ Barsotti-Tate groups in 2014, and Marco d'Addezio for several enlightening discussions on semistable Barsotti-Tate groups in 2022. K.Z. thanks A. Langer for discussions on elementary transformations of de Rham bundles and A. Javenpeykar for discussions on the degree of automorphic bundles. All three authors would further like to acknowledge the debt of this work to the recent work of Snowden-Tsimerman \cite{snowden2018constructing}. J.Y. is supported by National Natural Science Foundation of China Grant No. 12201595, the Fundamental Research Funds for the Central Universities and CAS Project for Young Scientists in Basic Research Grant No. YSBR-032.
\end{acknowledgement*}

\end{document}